\documentclass[12pt, reqno]{amsart}
\usepackage{ifthen,amsfonts,
amsmath,amssymb, enumerate,
epsfig,color,graphicx,cite,bbm,graphics}
\usepackage{stmaryrd}
\usepackage{caption}
\usepackage{a4wide}
\usepackage{xfrac}
\usepackage{mathrsfs,bm,amsthm,mathtools,yfonts}
\usepackage{xcolor}
\mathtoolsset{showonlyrefs,showmanualtags}

\makeatletter

\begingroup

\newtheorem{theorem}{Theorem}[section]
\newtheorem{lemma}[theorem]{Lemma}
\newtheorem{propos}[theorem]{Proposition}

\endgroup

\newtheorem{remark}[theorem]{Remark}

\def\RR{{\mathbb{R}}}

\def\TT{{\mathbb  T}}

\def\NN{{\mathbb{N}}}
\newcommand{\eps}{{\varepsilon}}

\newcommand{\Leb}[1]{{\mathscr L}^{#1}} % Misura di Lebesgue

\newcommand{\fR}{{\mathcal R}}
\newcommand{\supp}{{\rm supp}}

\renewcommand{\div}{{\text {div}}\,}

\DeclareMathOperator{\curl}{curl}

\makeindex
\begin{document}

\title[]
{Nonuniqueness of solutions to the Euler equations with vorticity in a Lorentz space}
\author{Elia Bru\'e}
\address[Elia Bru\'e]{
\newline \indent Institute for Advanced Study
\newline \indent 1 Einstein Dr., Princeton NJ 05840, USA}
\email{elia.brue@ias.edu}
\author{Maria Colombo}
\address[Maria Colombo]{
\newline \indent EPFL B, Station 8
\newline \indent CH-1015 Lausanne, CH}
\email{maria.colombo@epfl.ch}
%\author{Camillo De Lellis}
%\address[Camillo De Lellis]{
%\newline \indent School of Mathematics, Institute for Advanced Study and Universit\"at Z\"urich 
%\newline \indent 1 Einstein Dr., Princeton NJ 05840, USA}
%\email{camillo.delellis@math.ias.edu}

\begin{abstract}
For the two dimensional Euler equations, a classical result by Yudovich states that solutions are unique in the class of bounded vorticity; it is a celebrated open problem whether this uniqueness result can be extended in other integrability spaces. We prove in this note that such uniqueness theorem fails in the class of vector fields $u$ with uniformly bounded kinetic energy and  vorticity in the Lorentz space $L^{1, \infty}$.

\end{abstract}
\footnote{
 %35Q30 (76D03)
}
\maketitle

%\tableofcontents

\section{Introduction}
Let us consider the $2$-dimensional Euler equation
\begin{equation}
\label{eqn:EU}
\begin{dcases}
\partial_t u + \div (u \otimes u ) + \nabla p =0 \\
\div u = 0
\end{dcases}
\end{equation}
where $u : [0,1] \times \TT^2 \to \RR^2$ is the velocity of a fluid and $p : [0,1] \times \TT^2 \to \RR$ the pressure. %$u\in C^0([0,1]; L^2(\TT^2))$ and $\omega := \curl( u) \in C^0([0,1]; L^{1,\infty}(\TT^2))$.
This system can be equivalently rewritten as the two dimensional Euler system in vorticity formulation, which is a transport equation for the vorticity $\omega = {\rm curl} (u)$, i.e.
\begin{equation}\label{e:2deulervort}
\left\{
\begin{array}{l}
\partial_t \omega + u \cdot \nabla \omega=0  \\ 
u = \nabla^\perp \Delta ^{-1} \omega\, .
\end{array}\right.\qquad\mbox{in } \TT^2 \times [0,1] \, .
\end{equation}
% While the literature on this topic is huge both from the PDE point of view and with respect to typical fluid questions, we focus here only on the milestones in the existence and uniqueness theory.
In the latter formulation it is clear that $L^p$ norms of the vorticity are formally conserved for any $p \in[1,\infty]$. For $p>1$, this was used in \cite{DPM87} to prove existence of distributional solutions starting from an initial datum with vorticity in $L^p$. A similar existence result is much more involved for $p=1$, and it was obtained by Delort \cite{Delort91} (see also \cite{DPM87,EvansMuller94}), improving the existence theory up to measure initial vorticities in $H^{-1}$ (this latter condition guarantees finiteness of the energy) whose positive (or negative) part is absolutely continuous.  
As regards uniqueness, the classical result of Yudovich  \cite{Yud62,Yud63} (see also the proof in \cite{MR2246357}) %assumes the initial vorticity to be bounded: it 
states %more precisely 
that, given an initial datum $\omega_0 \in L^\infty$, there exists a unique bounded solution to \eqref{e:2deulervort} starting from $\omega_0$.
%Loeper \cite{MR2246357} provided later a simple and flexible proof of Yudovich Theorem based on optimal transport.
%
%However, the classical problem raised by Yudovich about the sharpness of his result is still open. It is even unknown whether uniqueness holds in the class of $L^2$ velocity with vorticity in $L^\infty(L^p)$, $p\in [1,\infty)$.
%In this paper we show ill-posedness in the Lorentz class $L^{1,\infty}$.
However, the classical problem raised by Yudovich about the sharpness of his result is still open. 
Let $u_0$ be an initial datum in $L^2$ with $\curl u_0$ in some function space $X$. Is the solution of the Euler equation in vorticity formulation unique in the class $L^\infty(X)$?

%Yudovich Theorem guarantees that for $X=L^\infty$, the answer to the previous questions is positive. 
The main result of this paper provides a negative answer when $X$ is the Lorentz space $L^{1,\infty}$.
\begin{theorem}\label{thm: main}
	There exists a nontrivial  solution $u\in C^0([0,1]; L^2(\TT^2))$ to \eqref{eqn:EU} satisfying
	\begin{itemize}
		\item[(i)]  $\omega = \curl(u)\in C^0([0,1]; L^{1,\infty}(\TT^2))$;
		%		 in particular $\omega\in C^0([0,1]; L^{1,\infty}(\TT^2))$ ;
		\item[(ii)] $u(0,\cdot) = 0$.		
	\end{itemize}
\end{theorem}

Recently, there have been formidable attempts to disprove this conjecture for $X= L^p$, none of which has by now fully solved it.
Vishik \cite{Vis18,Vis18a} proposed a complex line of approach to this problem, which however has the price of showing nonuniqueness only with an additional degree of freedom, namely a forcing term in the right-hand side of the equation \eqref{e:2deulervort} in the integrability space $L^1(L^p)$. 
The nonuniqueness suggested by this work is of symmetry breaking type%; in other words, the initial datum has a radial symmetry but, in addition to the radially symmetric solution, another one appears which looses this property. 
and, in contrast with the ideas of this paper, his nonuniqueness stems from the linear part of the equation, by carefully choosing an initial datum that sees the instability directions of a linearized operator.

A second attempt has been pursued by Bressan and Shen \cite{BrSh21}, based on numerical experiments which share the symmetry breaking type of nonuniqueness of Vishik. Their work is a first step in the direction of a computer assisted proof.

{Our approach is instead of different nature and stems from the convex integration technique}. The latter was introduced by De Lellis and Sz\'ekelyhidi \cite{DeLellisSzekelyhidi09} in the context of nonlinear PDEs, inspired by the work of Nash on isometric embeddings \cite{Nash}, which found striking applications in recent years to different PDEs (see for instance  \cite{Isett2018Annals,BSV16,MR3352460,MoSz2019AnnPDE,MR4138227,Buckmaster:2021um} and the references quoted therein). 
 As such, our proof would probably be less constructive with respect to the strategies of  \cite{Vis18,Vis18a} and \cite{BrSh21}, where an initial datum for which nonuniqueness is expected is described fairly explicitly as well as the mechanism for the creation of two different singularities. Conversely, the latter approaches see the drawbacks described above and are by no means ``generic'' in the initial data, whereas it is known (see for instance \cite{DRT19typicality,CDRS21typicality}) that convex integration methods yield not only the lack of uniqueness/smoothness for certain specific initial data, but also that solutions are typical (in the Baire category sense). 

%
%\bigskip
%
%
%
%	
%	Recently, {\em Buckmaster and Vicol \cite{BV} found a way to implement a convex integration scheme to the Navier-Stokes system, by introducing the idea of intermittency}: the perturbations are not only fastly oscillating, but also supported on small sets in space. They construct a large class of irregular distributional solutions of the Navier-Stokes equation, with prescribed kinetic energy.
%	
%}
\subsection{Strategy of proof}

The guiding thread of this construction 
%with several fundamental differences in the various instances, 
is an iterative procedure, where one starts from a solution $(u_0, p_0,R_0)$ of the Euler  equations
%(and analogously for Navier-Stokes) 
with an error term in the right-hand side, namely
\begin{equation}
	\label{eqn:EU-R}
	\begin{dcases}
		\partial_t u + \div (u \otimes u ) + \nabla p = \div R \\
		\div u = 0,
	\end{dcases}
\end{equation}
and iteratively corrects this error by adding a fastly oscillating perturbation to the approximate solution. The nonlinear interaction of this perturbation with itself generates a resonance which allows for the cancellation of the previous error; the other terms are mainly seen as new error terms, with smaller size with respect to the previous error.  
More precisely, we define the new solution $(u_1, p_1, R_1)$ by setting
$$
u_1 =u_0+a w_{\lambda} 	\, ,
\quad
w_\lambda(x) := w(\lambda x) \, \, \, \lambda\in \mathbb{Z} \, ,
$$
where $\lambda \gg 1$ is a higher frequency with respect to the typical frequencies in $u_0$,
$w$ is called building block of the construction and enjoys suitable integrability properties, 
%$w_{\lambda} $ is a $1/\lambda$ periodic function (called building block in the following) 
$a$ is a slowly varying coefficient. The cancellation of error happens because the low frequency term in $a^2 w_{\lambda}  \otimes w_{\lambda}$ satisfies
$$
a^2 \int_{\TT^2} w_{\lambda}  \otimes w_{\lambda}  \sim R \, .
$$
This forces us to require that $\int_{\TT^2} |w|^2 =\int_{\TT^2} |w_{\lambda}|^2 \sim 1$.  On the contrary, we wish to control the quantity $\|D u_1\|_{X}$ and for this end we need $\|D w_{\lambda}\|_{X} $ arbitrarily small. This imposes us a restriction on the space $X$
since the Sobolev inequality in Lorentz spaces (see \cite{MR438106}) states that
   \begin{equation}\label{in:alvino}
   	\| u \|_{L^{p^*,q}} \le C(p,q,d) \| \nabla u \|_{L^{p,q}}
   	\quad \text{for $p\in [1,d)$, $q\in [1,\infty]$ and } p^* = dp/(d-p) \, ,
   \end{equation}
giving that
\begin{equation*}
	\| \nabla w_\lambda \|_{L^{1,2}}
	=
	\lambda \| \nabla w \|_{L^{1,2}}
	\gtrsim \lambda  \| w \|_{L^2} 
	\sim \lambda
	\gg 1 \, ,
\end{equation*}
when applied with $p=1$ and $q=2$. In particular, 
%   More precisely the inequality~\eqref{in:alvino} with $p=1$ e $q=2$, shows that, 
with the current method of proof (and in particular with the current way to cancel the error in the iteration), $X= L^1$ or $X= L^{1,2}$ are not allowed; only $X = L^{1,q}$ for  $q>2$ could be obtained. To avoid technicalities, we present the proof with $X= L^{1,\infty}$.

The main novelty in the proof of Theorem~\ref{thm: main} regards the construction of a new family of building blocks. 
They are designed as a bundle of almost solutions to Euler, suitably rescaled and periodized in order to saturate the $L^{1,\infty}$ norm. To this aim we take advantage of intermittent jets, introduced in \cite{BCV18}, and we bundle them in a similar spirit to the atomic decomposition of Lorentz functions. A challenge is to keep different building blocks disjoint in space-time, since we work in two dimensions and since each component of the bundle has its own characteristic speed. We refer the reader to Section \ref{sec:build} for the precise construction and more explanations on our choice of building blocks.

\begin{remark}
The proof of Theorem~\ref{thm: main} is flexible enough, due to the exponential convergence of the iterative sequence, to give $\omega \in L^{1,q}$ for some $q \gg 1$. A technical refinement of the current proof, based on Remark \ref{rmk:improved Lorentz}, would give $q>4$.
\end{remark}

\begin{remark}
A fractional version of the inequality~\eqref{in:alvino}, namely
	\begin{equation}
		\| u \|_{L^{p^*,q}} \le C(p,q,d,s) \| D^s u \|_{L^{p,q}}
		\quad \text{for $s\in (0,1)$, $p\in [1,ds)$, $q\in [1,\infty]$ and } p^* = dp/(d-ps) \, ,
	\end{equation}
   gives that $\| D^s u \|_{L^r} \le \| D^{s-1} \omega \|_{L^r} \le \|  D^{s-1} \omega \|_{L^{\frac{2}{1+s},\infty}}
   \le C \| \omega \|_{L^{1,\infty}} < \infty$,
%   \begin{equation}
%   	\| D^s u \|_{L^r} \le \| D^{s-1} \omega \|_{L^r} \le \|  D^{s-1} \omega \|_{L^{\frac{2}{1+s},\infty}}
%   	\le C \| \omega \|_{L^{1,\infty}} < \infty \, ,
%   \end{equation}
 hence  the vector field $u$ built in Theorem~\ref{thm: main} enjoys the further fractional regularity 
$$
u \in C^0([0,1]; W^{s,\frac{2}{1+s}-\eps}(\TT^2))\qquad \mbox{ for any }s\in (0,1) \mbox{ and }\eps >0 \, .
$$ 
\end{remark}

\subsection*{Acknowledgments}
EB was supported by the Giorgio and Elena Petronio Fellowship at the Institute for Advanced Study. MC was supported by the SNSF Grant 182565. The author wish to thank Camillo De Lellis for interesting discussions on the theme of the paper.

\section{Iteration and Euler-Reynolds system} 
We consider the system of equations  \eqref{eqn:EU-R} in $[0,1] \times \TT$,
%\begin{equation}\label{eqn:EU-R}
%\begin{cases}
%\partial_t u + \div(  u \otimes u) + \nabla p=-\div R
%\\ \\
%\div u=0.
%\end{cases}
%\end{equation}
where $R$ is a traceless symmetric tensor.

As already remarked, our solution to \eqref{eqn:EU} is obtained by passing to the limit solutions of \eqref{eqn:EU-R} with suitable constraints on $u$ and $R$. The latter are built by means of an iterative procedure based on the following.

\begin{propos}\label{prop: iterative}
	There exists $M>0$ such that the following holds. For any smooth solution $(u_0, p_0, R_0)$ of \eqref{eqn:EU-R}, there exists another smooth solution $(u_1,p_1,R_1)$ of \eqref{eqn:EU-R} such that
	\begin{itemize}
		\item[(i)] $\| R_1 \|_{L^\infty(L^1)} \le \frac{1}{3}\| R_0 \|_{L^\infty(L^1)}$ ;
		
		\item[(ii)] $\| u_1 - u_0 \|_{C^0(L^2)} + \|D(u_1 - u_0) \|_{C^0(L^{1,\infty})} \le M \| R_0 \|_{L^\infty(L^1)}$;
%		 and $\|D(u_1 - u_0) \|_{C^0(L^{1,\infty})} \le \eps$ ;
		\item[(iii)] if $R_0(t,\cdot )=0$ in $[0,t_0]$, then $R_1(t,\cdot )=0$ and $u_1(t,\cdot) = u_0(t,\cdot)$ in $[0,t_0/2]$.
	\end{itemize}	
\end{propos}

\begin{proof}[Proof of Theorem \ref{thm: main} given Proposition \ref{prop: iterative}]

	Fix $\lambda>0$. We start the iteration scheme with 
	\begin{equation}
		u_0(t,x) := \chi(t) \sin(x_2 \lambda) e_1
	\end{equation}
    where $\chi\in C_c^\infty([0,1])$, $\chi=0$ in $[0, 1/2]$ and $\chi = 1$ in $[3/2, 1]$.
    Notice that $-\div R_0 = \chi'(t) \sin(x_2 \lambda) e_1 + \nabla p$, hence we can choose a traceless symmetric tensor $R_0$ such that $\| R_0 \|_{L^1} \le C\lambda^{-1}$.
    
    Applying iteratively Proposition \ref{prop: iterative} with $t_0 = 1/2$ we build a sequence $\{(u_n, p_n, R_n)\, : \, n\in \NN\}$ of smooth solutions to \eqref{eqn:EU} such that, for any $n\ge 0$, it holds
    \begin{equation}
    	\| R_n \|_{L^\infty(L^1)} \le C 3^{-n} \lambda^{-1} \, , \quad
    	\| u_{n+1} - u_n \|_{C^0(L^2)} + \| D(u_{n+1} - u_n) \|_{C^0(L^{1,\infty})} \le CM 3^{- n + 1} \lambda^{-1} \, ,
%    	\quad
%    	\| D(u_{n+1} - u_n) \|_{C^0(L^{1,\infty})} \le 3^{-n} \, ,
    \end{equation}
    and $u_n(t,\cdot) = 0$ for any $t\in [0,2^{-n-1}]$.
    It follows that $R_n \to 0$ in $L^\infty(L^1)$ and  $u_n \to u$ in $C^0(L^2)$, where $u$ satisfies the assumptions of Theorem \ref{thm: main}. 
    To prove that $Du\in C^0(L^{1,\infty})$, a bit of extra care is needed since only the weak triangle inequality
    $
    \| f + g \|_{L^{1,\infty}} \le 2 \| f\|_{L^{1,\infty}} + 2 \| g \|_{L^{1,\infty}}
    $
    holds true.
    However, the latter is enough for our purposes
    \begin{align*}
    	\| D u_N \|_{C^0(L^{1,\infty})} 
    	& = 
    	\| Du_0 + D(\sum_{n=0}^{N-1} u_{n+1} - u_n)\|_{C^0(L^{1,\infty})}
    	\\& \le
    	2 \| Du_0 \|_{C^0(L^{1,\infty})} + \sum_{n=0}^{N-1} 2^{n + 1} \| D(u_{n+1} - u_n) \|_{C^0(L^{1,\infty})}
    	\\& \le 2 \| Du_0 \|_{C^0(L^{1,\infty})} + CM\lambda^{-1}\sum_{n=0}^{N-1} 2^{n + 1}3^{-n + 1}
    	< \infty \, .\qedhere
    \end{align*}
\end{proof}

The remaining part of this note is devoted to the proof of Proposition \ref{prop: iterative}. In Section \ref{sec:build} we introduce the building blocks of our construction, in Section \ref{sec:def perturbation} we use them to define the perturbation $u_1-u_0$, finally in Section \ref{sec:new error}, we introduce the new error term $R_1$ and show that it can be made arbitrarily small.

\section{Preliminary lemmas}
\subsection{Lorentz spaces}
For every measurable function $f:\TT^d \to \RR$ we recall the definition
\begin{equation*}
	\|{f}\|_{L^{r,q}} := r^{1/q}
	 \big\|  \lambda \Leb{d}(\{  |f| \ge \lambda \})^{1/r} \|_{L^q((0,{\infty}), \frac{\, d \lambda}{\lambda}  )} \, ,
\end{equation*}
(see e.g. \cite{Grafakos})
and we define the Lorentz space $L^{r,q}$ with $r\in [1,\infty)$, $q\in [1,\infty]$, as the space of those functions $f$ such that $\|f\|_{L^{r,q}}<\infty$. Note that, in spite of the notation, $\|\cdot\|_{L^{r,q}}$ is in general not a norm but for $(r,q)\neq (1,\infty)$ the topological vector space $L^{r,q}$ is locally convex and there exists a norm $|||\cdot|||_{r,q}$ which is equivalent to $\|\cdot\|_{L^{r,q}}$ in the sense that the inequality $C^{-1} |||f|||_{r,q}\leq \|f\|_{L^{r,q}} \leq C |||f|||_{r,q}$ holds.

\subsection{Improved H\"older inequality}
We recall the following improved H\"older inequality, stated as in \cite[Lemma 2.6]{MoSz2019AnnPDE} (see also \cite[Lemma 3.7]{BV}).
If $\lambda \in \NN$ and $f,g:\TT^2 \to \RR$ are smooth functions, then we have 
\begin{equation}
	\label{eqn:impr-holder}
	\| f(x) g(\lambda x) \|_{L^p} \leq \| f \|_{L^p} \| g \|_{L^p} + C(p) \lambda^{-1/p}  \|f \|_{C^1} \|g\|_{L^p} \, .
\end{equation}
When $\int_{\TT^2} g= 0$, then
\begin{equation}
\label{eqn:l26}
\Big| \int_{\TT^2} f(x) g(\lambda x) \, dx \Big| 
\leq\Big| \int_{\TT^2} f(x) \Big(g(\lambda x) - \int g \Big) \, dx \Big| + \Big| \int_{\TT^2} f dx \Big |\cdot \Big| \int_{\TT^2} g dx \Big| 
\leq C \lambda^{-1} \|f \|_{C^1} \|g\|_{L^1} \, . 
%+ \Big| \int f \Big| \cdot \Big | \int g \Big| .
\end{equation}

\subsection{Anti-divergence operators}
Let now us introduce the anti-divergence operator
\begin{equation}
	\mathcal{R}_0 : C^\infty(\TT^2;\RR^2) \to C^\infty(\TT^2;\text{Sym}_2)\, ,
	\quad\quad 
	\mathcal{R}_0 v := (D\Delta^{-1} + (D \Delta^{-1})^{T} - I \cdot \div \Delta^{-1} )\Big(v - \int_{\TT^2}v\Big) \, .
\end{equation}
Here $\text{Sym}_2$ denotes the space of symmetric matrices in $\mathbb{R}^2$.
It is simple to check that $\div(\mathcal{R}_0(v)) = v-\int_{\TT^2} v$, and that $D \mathcal{R}_0$ is a  Calderon-Zygmund operator, in particular it holds
\begin{equation}\label{eq: pre1}
	\| \mathcal{R}_0(v) \|_{L^p} \le C \| \Delta^{-1/2} v \|_{L^p}
	\quad
	\text{for any $p\in (1, \infty)$} \, ,
\end{equation}
\begin{equation}\label{eq: pre2}
	\| \mathcal{R}_0(v) \|_{L^p} \le C(p) \| v \|_{L^p}
	\quad
	\text{for any $p\in [1, \infty]$} \, .
\end{equation}
Notice that \eqref{eq: pre1} and \eqref{eq: pre2} allow showing that
	\begin{equation}\label{eq: pre3}
		\| \mathcal{R}_0(v_\lambda) \|_{L^p} \le C(p) \lambda^{-1} \| v \|_{L^p}
		\quad
		\text{for any $p\in [1, \infty]$} \, ,
	\end{equation}
where $v_\lambda(x) := v(\lambda x)$ for some $\lambda\in \NN$. The latter is immediate for $p\in (1,\infty)$, since
\begin{equation}
	 \| \mathcal{R}_0(v_\lambda) \|_{L^p} \le C \| \Delta^{-1/2} v_\lambda \|_{L^p} \le C \lambda^{-1} \| v \|_{L^p} \, ,
\end{equation}
in the case $p=1$ and $p=\infty$ we need to take advantage of the Sobolev embedding theorem:
\begin{equation}
	\| \mathcal{R}_0(v_\lambda) \|_{L^1}
	\le \| \mathcal{R}_0(v_\lambda) \|_{L^{3/2}}
	 \le C \| \Delta^{-1/2} v_\lambda \|_{L^{3/2}} 
	 \le C \lambda^{-1} \| \Delta^{-1/2} v \|_{L^{3/2}}
	 \le C \lambda^{-1} \| v \|_{L^1} \, .
\end{equation}

\begin{lemma}
\label{lemma23}
	Let $\lambda \in \NN$ and $f \in C^\infty( \TT^2; \RR)$, $v\in C^\infty( \TT^2; \RR^2)$ with $\int v = 0$, and $v_\lambda= v(\lambda x)$. If we set 
	$$
	\fR (f  v_\lambda) = f  \mathcal{R}_0 v_\lambda -\mathcal{R}_0 (\nabla f \cdot \mathcal{R}_0 v_\lambda+\int f v_\lambda) \in C^\infty(\TT^2; \text{Sym}_2) \, ,
	$$ 
	then we have that $\div  \fR (f v_\lambda) = f v_\lambda-\int_{\TT^2} f v_\lambda$ and
	\begin{equation}
	\label{ts:antidiv}
	\| \fR (f  v_\lambda)\|_{L^p} \leq C(p) \lambda^{-1} \|f\|_{C^{1}} \| v\|_{L^p} \qquad \mbox{for every }\,  p\in [1,\infty] \, .
	\end{equation}
%    \begin{equation}
%    	\label{ts:antidiv}
%    	\|D^k \fR (f  g_\lambda)\|_{L^p} \leq C(k,p) \lambda^{k-1} \|f\|_{C^{k+1}} \| g\|_{W^{k,p}} \qquad \mbox{for every } k\in \NN, p\in [1,\infty].
%    \end{equation}
\end{lemma}
\begin{proof}
	The verification of $\div  \fR (f v_\lambda) = f v_\lambda-\int_{\TT^2} fv_\lambda$ is immediate. To prove \eqref{ts:antidiv} we use \eqref{eq: pre3} and \eqref{eqn:l26}:
	\begin{equation}
		\| f \mathcal{R}_0 v_\lambda\|_{L^p}
		\le \| f \|_{C^0} \| \mathcal{R}_0 v_\lambda\|_{L^p}
		\le C \lambda^{-1} \| f \|_{C^0} \| v\|_{L^p} 
		\, ,
	\end{equation}
	\begin{equation*}
		\| \mathcal{R}_0 (\nabla f \cdot \mathcal{R}_0 v_\lambda+\int_{\TT^2} f v_\lambda) \|_{L^p}
		\le 
		C \| \nabla f \cdot \mathcal{R}_0 v_\lambda+\int_{\TT^2} f v_\lambda \|_{L^p}
		\le C\lambda^{-1} \| f \|_{C^1} \| v \|_{L^p} + C\lambda^{-1} \| f\|_{C^1} \| v \|_{L^1} \, .\qedhere
	\end{equation*}	
%	It is enough to combine \cite[lemma 3.5]{MoSa2019} and the remark in \cite[page 12]{MoSa2019}.	
\end{proof}

\begin{remark}\label{rmk:scalar}
The operator  $\mathcal{R}$ can be also defined on scalar functions $f : \TT^2 \to \RR$, $v: \TT^2\to \RR$ as
\begin{equation}\label{eq:R scalar}
	\fR (f  v_\lambda) = f  \nabla \Delta^{-1} v_\lambda -\nabla \Delta^{-1}\Big(\nabla f \cdot \mathcal{R}_0 v_\lambda+\int f v_\lambda \Big) \in C^\infty(\TT^2; \RR^2) \, , 
\end{equation}
and arguing as in Lemma \eqref{lemma23} we can easily show that $\div  \fR (f v_\lambda) = f v_\lambda-\int_{\TT^2} f v_\lambda$ and
\begin{equation*}
	\| \fR (f  v_\lambda)\|_{L^p} \leq C(p) \lambda^{-1} \|f\|_{C^{1}} \| v\|_{L^p} \qquad \mbox{for every }\,  p\in [1,\infty] \, .
\end{equation*}
\end{remark}

\begin{lemma} For any $a\in C^\infty(\TT^2)$ and $A\in C^\infty(\TT; \RR^{2\times 2})$ with $\int_{\TT^2} A =0$, it holds
	\begin{equation}\label{eqn:R-di-div}
		\|\mathcal{R}_0\mathcal{R}(\nabla a \cdot \div A)\|_{L^1} \le C(\|a \|_{C^3}) \| A \|_{L^1} \, .
	\end{equation}	
\end{lemma}
\begin{proof}
	Set $T(A) := \mathcal{R}(\nabla a \cdot \div A)$.
%	Observe that the adjoint operator of $A \to \nabla \Delta^{-1}\mathcal{R}(\nabla a, \div A)$ is given gy
%	\begin{equation}
%		
%	\end{equation}
	By duality, it suffices to show that
	\begin{equation}
			\| T^* \mathcal{R}_0^*(B)\|_{L^\infty} \le C(\|a \|_{C^3}) \| B \|_{L^\infty} \, ,
			\end{equation}
%	\begin{equation}
%		\|D (\nabla a\cdot \mathcal{R}^*\mathcal{R}^* B)\|_{L^\infty} \le C(\|a \|_{C^3}) \| B \|_{L^\infty} \, ,
%	\end{equation}
	where $T^*$ and $\mathcal{R}_0^+$ denote the adjoint of $T$ and $\mathcal{R}_0$, respectively. To this aim we employ the Sobolev embedding and the fact that $D T^* \mathcal{R}_0^*(B)$ maps $L^p$ into $L^p$ for any $p\in (1,\infty)$:
	\begin{equation*}
		\| T^* \mathcal{R}_0^*(B) \|_{L^\infty} 
		\le
		C \| D T^* \mathcal{R}_0^*(B)\|_{L^3}
		\le 
		C(\|a \|_{C^3}) \| B \|_{L^3}
		\le 	C(\|a \|_{C^3}) \| B \|_{L^\infty} \, .\qedhere
	\end{equation*}
\end{proof}

\section{Building blocks}\label{sec:build}
%The following proposition provides the building blocks of our construction.

In this section we introduce the building blocks of our construction. They will be employed in Section \ref{sec:def perturbation} to define the principal term of $u_1 - u_0$ in Proposition \ref{prop: iterative}.

\begin{propos}[Building blocks] \label{lemma: building blocks}
	Set $\xi_1 := e_1$, $\xi_2:= e_2$, $\xi_3 := e_1 + e_2$ and $\xi_4 : = e_1 - e_2$.
	Then, for any $\eps>0$ there exist $W^p_{i}, W^c_{i}, Q_i \in C^\infty((-1,1)\times \TT^2; \RR^2)$, $A_i \in C^\infty((-1,1)\times \TT^2; \text{Sym}_2)$ for $i=1,\ldots, 4$, such that
	\begin{itemize}
		
		\item[(i)] $\div (W_i^p + W_i^c) =0$, $\partial_t Q_i = \div(W_i^p\otimes W_i^p)$, and $\partial_t (W^p_i + W^c_i) = \div(A_i)$ ;
		
		\item[(ii)] $\int_{\TT^2} A_i =0$, $\int_{\TT^2} W_i^p = \int_{\TT^2} W^c_i = 0$, and $W_i^p$, $W_i^c$, $A_i$ are $\lambda^{-1}$-periodic functions for some $\lambda\in \mathbb{Z}$ with $\lambda \ge \eps^{-1}$;
		%		\begin{equation}
		%			\int W_i^p \otimes W_i^c = \xi_i \otimes \xi_i \, ;
		%		\end{equation}
		\item[(iii)] $\int_{\TT^2} W_i^p \otimes W_i^p = \frac{\xi_i}{|\xi_i|} \otimes \frac{\xi_i}{|\xi_i|}$;
		
		\item[(iv)] the following estimates hold
		\begin{equation}
			\eps \| W_i^p \|_{L^2} + \| W_i^p \|_{L^1} + \| W_i^c \|_{L^2}  \le \eps \, ,
		\end{equation}
		\begin{equation}
			\| D (W_i^p + W_i^c) \|_{L^{1,\infty}} + \| Q_i \|_{L^2} + \| D Q_i \|_{L^{1,\infty}} + \| A_i \|_{L^1} < \eps \, ;
		\end{equation}

		\item[(v)] for $i\neq i'$ the union of the supports of $W^p_i$, $W_i^c$, $Q_i$, is disjoint in space-time from the union of the supports of $W^p_{i'}$, $W_{i'}^{c}$, $Q_{i'}$.
	\end{itemize}	
\end{propos}

The velocity field $W^p_i$ is the principal term, it has zero mean, high frequency $\lambda\ge \eps^{-1}$, is controlled in the relevant norms (cf. (iv)), and satisfies the fundamental property (iii): the quadratic interaction $W_i^p\otimes W_i^p$ produces the lower order term $\frac{\xi_i}{|\xi_i|} \otimes \frac{\xi_i}{|\xi_i|}$. The latter, combined with slow coefficients $a_i\in C^\infty(\TT^2)$, is used to cancel the error $R_0$ out.
To achieve the crucial bound $\| D W_i^p \|_{L^{1,\infty}}$ we design the principal term as
\begin{equation}\label{zzz}
	W_i^p(x,t) =W_{\xi_i,K,n_0}^p(t,x) := \frac{1}{K^{1/2}} \sum_{k=n_0 + 1}^{K + n_0} W^k_{(\xi_i)}(t, x) \, ,
\end{equation}
%\begin{equation}\label{zzz}
%	W_i^p(x,t) := \frac{1}{K^{1/2}} \sum_{k = n_0+1}^{K + n_0} W_k(x,t) \, ,
%\end{equation}
where $K, n_0 \gg 1$ are big parameters and $\xi_i $ is one of the four directions appearing in the statement of Proposition~\ref{lemma: building blocks}. {In a first stage, we build $W_i^p(x,t)$ for a fixed parameter $i$, ignoring the issue that, for different parameters, such functions will not have disjoint support as requested in Proposition~\ref{lemma: building blocks} (v); only in Section~\ref{subsec:combinatorial} we make sure to suitably time-translate them, making substantial use of their special structure, to guarantee that Proposition~\ref{lemma: building blocks} (v) holds .}
 The vector fields $W_k(x,t)$, $k=n_0+1, \ldots, n_0 + K$, are the $2$-dimensional counterpart of the intermittent jets introduced in \cite{BCV18}. They have $L^2$ norm equal to $1$, and are supported on disjoint balls of radius $2^{-k} r$, for some $r\ll 1$, which move in direction $e_i$ with speed $\mu 2^k$, where $\mu \gg 1$. The fast time translation is used to make $W_k$ 
``almost divergence free'' and ``almost solutions to the Euler equation''. 
In more rigorous terms, it means that there exist vector fields $W_k^p$, $Q_k$, that are smaller than $W_k$ satisfying
$\div(W_k + W_k^p)=0$ and $\partial_t Q_k = \div(W_k \otimes W_k)$. 
The vector fields $W_i^p$ and $Q_i$ are defined bundling together $W_k^p$ and $Q_k$ as we did in \eqref{zzz}.

Another important property we need is
that $W_i \otimes W_j =0$ when $i\neq j$. It is ensured by (iv) in Proposition \ref{lemma: building blocks}, which builds upon a delicate combinatorial lemma presented in section \ref{subsec:combinatorial}.

We finally explain the role of the matrix $A_i$ in our construction. Let us begin by noticing that the principal term $W_i^p$ has big time derivative, being fast translating in time. Hence, the term $\partial_t W_i^p$ cannot be treated as an error. To overcome this difficulty we impose an extra structure on $W^p_i$ and $W^c_i$. We construct them in order to have the identity $\partial_t(W_i^p + W_i^c) = \div(A_i)$, for some symmetric matrix $A_i$ which has small $L^1$-norm. The latter can be added to the new error term $R_1$.

\subsection{General notation}
Given a velocity field $u:=(u_1, u_2): \RR^2 \to \RR^2$ we write
	\begin{equation}
		u^{\perp} := (-u_2, u_1)\, ,
		\quad
		\text{curl}(u) := \partial_{1} u_2 - \partial_{2} u_1 \,
		\quad
		\text{div}(u) := \partial_1 u_1 + \partial_2 u_2\, .
	\end{equation}

Let us fix $r_\perp \ll r_\parallel \ll 1$ and $k\in \NN$. We adopt the following convention: given any $\rho:  \mathbb{R} \to \mathbb{R}$ supported in $(-1,1)$ we write
\begin{equation*}
	\rho_{r_\perp}^k(x):= \left(\frac{1}{2^{-k} r_\perp}\right)^{1/2} \rho\left( \frac{x - 2^{2-k}r_\perp}{2^{-k}r_\perp}  \right)\, ,
\end{equation*}

\begin{equation}
	\rho_{r_\parallel}^k(x):= \left(\frac{1}{2^{-k}r_{\parallel}}\right)^{1/2} \rho\left(\frac{ x }{2^{-k}r_{\parallel}}\right) \, .
\end{equation}
Notice that $\supp(\rho_{r_\perp}^k) \subset ( 3 \cdot 2^{-k} r_\perp, 5\cdot 2^{-k}r_\perp)$, in particular
\begin{equation}
	\supp (\rho^k_{r_\perp}) \cap \supp (\rho^{k'}_{r_\perp}) = \emptyset
	\quad
	\text{for $k\neq k'$\, ,}
\end{equation} 
and 
\begin{equation}
	\label{eqn:thin-tube1}
	\bigcup_{k \ge 1} \supp (\rho^k_{r_\perp})
	\subset (0, 5 r_\perp2^{-n_0}) \, .
\end{equation}
With a slight abuse of notation we keep denoting by $\rho_{r_\perp}^k, \rho_{r_\parallel}^k: \TT \to \RR$ their periodized version.

\subsection{Construction of the principal block}
We consider $\Phi, \psi: \RR\to \RR$ supported in $(-1,1)$, we set $\phi := -\Phi'''$ and assume $\int \psi^2 = \int \phi^2 =1$.
Given $r_\perp \ll r_\parallel \ll 1$ and $k\in \NN$ we have
\begin{equation}
	\supp (\phi^k_{r_\perp}) \cap \supp (\phi^{k'}_{r_\perp}) 
	= \supp ((\Phi')^k_{r_\perp}) \cap \supp ((\Phi')^{k'}_{r_\perp}) 
	=\supp ((\Phi'')^k_{r_\perp}) \cap \supp ((\Phi'')^{k'}_{r_\perp})
	= \emptyset
	\quad
	\text{for $k\neq k'$\, ,}
\end{equation} 
and 
\begin{equation}
\label{eqn:thin-tube}
	\bigcup_k \supp (\phi^k_{r_\perp}),\,  \bigcup_k \supp (\Phi^k_{r_\perp})
	\subset (0, 5 r_\perp2^{-n_0}) \, .
\end{equation}
We periodize $(\Phi')_{r_\perp}^k$, $(\Phi'')_{r_\perp}^k$, $\phi_{r_\perp}^k$, $\psi_{r_\parallel}^k$ keeping the same notation.

Given a vector $\xi \in \mathbb{Q}^2$, and parameters $\lambda, \mu \gg 1$ we set
\begin{equation}
	(\Phi')_{(\xi)}^k(x) := (\Phi')_{r_\perp}^k (   \lambda x\cdot \xi^{\perp} ) \, ,
	\quad
	(\Phi'')_{(\xi)}^k(x) := (\Phi'')_{r_\perp}^k (   \lambda x\cdot \xi^{\perp} ) \, ,
	\quad
	\phi_{(\xi)}^k(x) := \phi_{r_\perp}^k (  \lambda x\cdot \xi^{\perp} ) \, ,
\end{equation}

\begin{equation}
	\psi_{(\xi)}^k( x, t) := \psi_{r_\parallel}^k (   \lambda( x\cdot \xi + \mu 2^k t ) ) \, ,
\end{equation}

\begin{equation}
	W_{(\xi)}^k(x,t) := \frac{\xi}{|\xi|} \,  \psi_{(\xi)}^k(x, t) \phi_{(\xi)}^k(x) \, .
\end{equation}
We finally fix  $K, n_0 \in \NN$, and define the principal block
\begin{equation}\label{eqn:Wtutto}
	W_{\xi,K,n_0}^p(t,x) := \frac{1}{K^{1/2}} \sum_{k=n_0 + 1}^{K + n_0} W^k_{(\xi)}(t + t_k, x) \, ,
\end{equation}
where $t_k$ are time translations that will be chosen later.
The following fundamental identity holds
\begin{equation}\label{eq: 2}
	\int W_{\xi,K,n_0}^p \otimes W_{\xi,K,n_0}^p
	=
	\frac{1}{K}\sum_{k=n_0+1}^{K+n_0} \int W_{(\xi)}^k \otimes W_{(\xi)}^k  
	= 
	\frac{\xi}{|\xi|} \otimes \frac{\xi}{|\xi|} \int (\psi_{(\xi)}^k \phi_{(\xi)}^k)^2 
	=
	\frac{\xi}{|\xi|} \otimes \frac{\xi}{|\xi|} \, .
\end{equation}

%{\color{blue}Observe that
%\begin{equation}
%	W_{\xi,K,n_0}(t,x) = W_{\xi,K,n_0}(0,x + \xi t 2^k \mu) \, .
%\end{equation}
%}

\subsection{Correction of the divergence}
Observe that
\begin{equation}
	\div W_{(\xi)}^k (x,t) 
	= \frac{\lambda}{2^{-k}r_\parallel}  (\dot \psi)^k_{(\xi)} (x, t)  \phi^k_{(\xi)}(x)\, .
\end{equation}
Setting 
\begin{equation}
	(W^k_{(\xi)})^c(x,t) :=  \frac {r_\perp}{r_\parallel} \, \frac{\xi^{\perp}}{|\xi|} (\dot\psi)^k_{(\xi)} (x, t)  (\Phi'')^k_{(\xi)}(x)\, ,
\end{equation}
and using the identity $ 2^{-k} r_\perp \partial_{x_1} (\Phi'')_{r_{\perp}}^k = -\phi_{r_\perp}^k$ we get $\div (W_{(\xi)} + W^c_{(\xi)})=0$. 

To correct the divergence of $	W_{\xi, K ,n_0}$ we introduce
\begin{equation}
	W_{\xi, K ,n_0}^c(t,x) := \frac{1}{K^{1/2}} \sum_{k=n_0 + 1}^{K + n_0} (W^k_{(\xi)})^c(t + t_k, x) \, ,
\end{equation}
and set
\begin{equation}
	W_{\xi, K, n_0}(t,x) := W_{\xi, K ,n_0}^p(t,x) + W_{\xi, K ,n_0}^c(t,x) \, .
\end{equation}

\subsection{Time correction}
Let us now set
\begin{equation}
	Q_{(\xi)}^k(t,x):= \frac{1}{2^k\mu} \xi (\psi_{(\xi)}^k(x,t + t_k)\phi_{(\xi)}^k(x))^2 \, ,
\end{equation}
and observe that
\begin{equation}
\div(	W_{(\xi)}^k \otimes W_{(\xi)}^k) 
=
 2(W_{(\xi)}^k \cdot \nabla \psi_{(\xi)}^k) \phi_{(\xi)}^k \frac{\xi}{|\xi|}
= 
\frac{1}{2^k\mu}2 (W_{(\xi)}^k \cdot \partial_t \psi_{(\xi)}^k) \phi_{(\xi)}^k \frac{\xi}{|\xi|}
= 
\frac{1}{2^k\mu} \partial_t (\psi_{(\xi)}^k\phi_{(\xi)}^k)^2 \frac{\xi}{|\xi|} 
= 
\partial_t Q_{(\xi)}^k \, .
\end{equation}
Hence
\begin{equation}\label{eq: 3}
	\div(W_{\xi,K,n_0}^p \otimes W_{\xi,K,n_0}^p) 
	=
	 \frac{1}{K}\sum_{k=n_0+1}^{K+n_0} \div(	W_{(\xi)}^k \otimes W_{(\xi)}^k) 
	=
	\partial_t \left(   \frac{1}{K} \sum_{k=n_0 + 1}^{K + n_0} Q_{(\xi)}^k  \right) \, .
\end{equation}
The time corrector is defined as
\begin{equation}
	Q_{\xi,K,n_0}(t,x) := \frac{1}{K} \sum_{k=n_0 + 1}^{K + n_0} Q_{(\xi)}^k (t, x) \, .
\end{equation}

\subsection{Estimates on building blocks}

\begin{lemma}\label{lemma: first estimates}
	For any $N,M\ge 0$ integers and $p\in [1,\infty]$ there exists $C=C(N,M,p,|\xi|,\Phi,\psi)>0$ such that the following hold.
	\begin{equation}
		\| \nabla^N \partial_t^M \psi_{(\xi)}^k \|_{L^p(\TT)} \le C 2^{k(N+2M+1/2-1/p)} r_\parallel^{1/p - 1/2} \left( \frac{ \lambda}{r_\parallel} \right)^N  \left( \frac{ \lambda \mu}{r_\parallel} \right)^M \, ,
	\end{equation}

	\begin{equation}
		\| \nabla^N (\Phi')_{(\xi)}^k\|_{L^p(\TT)} 
		+ \| \nabla^N (\Phi'')_{(\xi)}^k\|_{L^p(\TT)}
		+	\| \nabla^N \phi_{(\xi)}^k\|_{L^p(\TT)}
		\le C 2^{k(N+1/2-1/p)} r_\perp^{1/p-1/2} \left( \frac{ \lambda}{r_\perp} \right)^N \, ,
	\end{equation}

	\begin{equation}
		\| \nabla^N \partial_t^M W_{(\xi)}^k \|_{L^p(\TT^2)} + \frac{r_{\parallel}}{r_\perp} \| \nabla^N \partial_t^M (W_{(\xi)}^k)^c \|_{L^p(\TT^2)}
		\le C 2^{k(N+ 2M +1 -2/p)} (r_{\parallel}r_\perp)^{1/p-1/2} \left(\frac{\lambda}{r_\perp}\right)^N \left( \frac{ \lambda \mu}{r_\parallel} \right)^M \, ,
	\end{equation}

	\begin{equation}
		2^k\mu \| \nabla^N \partial_t^M Q_{(\xi)}^k \|_{L^p(\TT^2)}
		\le 
		C2^{k(N + 2M + 2-2/p)}
		(r_{\parallel}r_\perp)^{1/p-1} \left(\frac{\lambda}{r_\perp}\right)^N \left( \frac{ \lambda \mu}{r_\parallel} \right)^M \, .
	\end{equation}
\end{lemma}
The proof of Lemma \ref{lemma: first estimates} is a simple computation, so we omit it. It implies the following, summing on $k$ and remininding that then terms in the sum in \eqref{eqn:Wtutto} have disjoint support,
 \begin{equation}\label{eq: W L2}
	\| W_{\xi,K,n_0}^p \|_{L^2(\TT^2)} + \frac{r_{\parallel}}{r_\perp} \| W_{\xi,K,n_0}^c \|_{L^2(\TT^2)}
	\le C \, 
\end{equation}
(in particular, this says that the principal part is much smaller than the corrector),
\begin{equation}\label{eq: Q L2}
		\| Q_{\xi,K,n_0} \|_{L^2(\TT^2)} \le \frac{C}{\mu (r_\parallel r_\perp)^{1/2}} \, ,
\end{equation}
and
\begin{equation}\label{eq: W Lp}
	\| W_{\xi,K,n_0}^p \|_{L^p(\TT^2)} + \frac{r_{\parallel}}{r_\perp} \| W_{\xi,K,n_0}^c \|_{L^p(\TT^2)}
	\le C \frac{(r_\perp r_\parallel)^{1/p-1/2}}{K^{1/2}} \, ,
	\quad \text{for any $p\in [1,2)$} \, .
\end{equation}

\begin{lemma}[Lorentz estimates]\label{lemma: Lorentz}
There exists $C=C(|\xi|, \Phi,\psi)>0$ such that
\begin{equation}
	\| D W_{\xi,K,n_0} \|_{L^{1,\infty}} \le C\frac{ \lambda}{K^{1/2}} \left(\frac{ r_\parallel}{r_\perp} \right)^{1/2}\, ,
\end{equation}
%	\begin{equation}
%		\| \omega_{\xi,K,n_0} \|_{L^{1,q}} \le \frac{C n_* \lambda_{q+1}}{K^{1/2-1/q}} \left(\frac{ r_\parallel}{r_\perp} \right)^{1/2}\, ,
%	\end{equation}
   \begin{equation}
   	\| D Q_{\xi,K,n_0} \|_{L^{1,\infty}} \le C \frac{\lambda}{\mu r_\perp K} \, .
   \end{equation}
\end{lemma}

\begin{proof}
	Observe that
	\begin{align*}
		|D W_{(\xi)}^k| & =	\lambda 2^k | r_\parallel^{-1} \frac{\xi}{|\xi|} \otimes \frac{\xi}{|\xi|} (\psi')^k_{(\xi)}(x,t)\phi^k_{(\xi)}(x)
		+  r_\perp^{-1} \frac{\xi}{|\xi|}\otimes \frac{\xi^\perp}{|\xi|} \psi^k_{(\xi)}(x,t)(\phi')^k_{(\xi)}(x)|
		\\& \le  \lambda 2^k r_\perp^{-1}  ( |(\psi')^k_{(\xi)}(x,t)||\phi^k_{(\xi)}(x)| + |\psi^k_{(\xi)}(x,t)||(\phi')^k_{(\xi)}(x)|   )
		\\& = \lambda \left(\frac{ r_\parallel}{r_\perp} \right)^{1/2}  \frac{1}{2^{-k}(r_\perp r_\parallel)^{1/2}}( |(\psi')^k_{(\xi)}(x,t)||\phi^k_{(\xi)}(x)| + |\psi^k_{(\xi)}(x,t)||(\phi')^k_{(\xi)}(x)|   )
		\\& := \lambda\left(\frac{ r_\parallel}{r_\perp} \right)^{1/2} \Omega_1^k(x,t) \, ,
	\end{align*}
and similarly
 \begin{align*}
    	| D Q_{(\xi)}^k | \le \frac{\lambda}{\mu r_\perp} \Omega_2^k(x,t) \, ,
    \end{align*}
        where for $i=1,2$
    \begin{equation}\label{z1}
    	|\Omega_i^k| \le C 2^{2k}(r_\perp r_\parallel)^{-1}\, ,
    	\quad
    	\Leb2(\supp (\Omega_i^k)) \le C 2^{-2k} r_\perp r_\parallel \, ,
    	\quad
    	\supp(\Omega_i^k)\cap \supp(\Omega_i^{k'})= \emptyset\, , \, \, \, \text{for $k\neq k'$} \, .
    \end{equation}
    Let us know fix $s\ge 1$ and $k_*$ the smallest integer satisfying $k_*\ge n_0+1$ and $C2^{2k^*} \ge s K^{1/2} r_\perp r_\parallel$. It holds
    \begin{equation}
    \Leb2\left(\left\lbrace\frac{1}{K^{1/2}} \sum_{k=n_0+1}^{K + n_0} \Omega_1^k \ge s \right\rbrace \right)
    =
    \sum_{k = n_0+1}^{K + n_0} \Leb2(\{\Omega_1^k \ge s K^{1/2}\})
    \le 
    \sum_{k=k_*}^{K + n_0} \Leb2(\{\Omega_1^k \ge s K^{1/2}\}) \, .
    \end{equation}
    From \eqref{z1} and the choice of $k^*$ we get
    \begin{align*}
    	\sum_{k=k_*}^{K + n_0} \Leb2(\{\Omega_k \ge s K^{1/2} \})
    	\le
    	\sum_{k=k_*}^{K+n_0} C 2^{-2k} r_\perp r_\parallel
    	\le  \frac{C}{ s K^{1/2}}\sum_{k\ge k^*} 2^{2k^*-2k} 
    	 \le \frac{C}{s K^{1/2}}\, ,
    \end{align*}
    hence
    \begin{equation}
    	\| D W_{\xi,K,n_0}^p \|_{L^{1,\infty}}
    	\le 
    	  \lambda\left(\frac{ r_\parallel}{r_\perp} \right)^{1/2} \|  \frac{1}{K^{1/2}} \sum_{k=n_0+1}^{K+n_0} \Omega^k \|_{L^{1,\infty}}
    	 \le C^2 \frac{\lambda}{K^{1/2}} \left(\frac{ r_\parallel}{r_\perp} \right)^{1/2}\, ,
    \end{equation}
   the estimate on $\| D W^c_{\xi,K,n_0} \|_{L^{1,\infty}}$ can be obtained following the same strategy.
   An analogous argument gives
   \begin{equation}
   	\Leb2\left(\left\lbrace\frac{1}{K} \sum_{k=n_0+1}^{K + n_0} \Omega_2^k \ge s \right\rbrace \right)
   	\le \frac{C}{s K} \, ,
%   	=
%   	\sum_{k = n_0+1}^{K + n_0} \Leb2(\{\Omega_1^k \ge t K^{1/2}\})
%   	\le 
%   	\sum_{k=k_*}^{K + n_0} \Leb2(\{\Omega_1^k \ge t K^{1/2}\}) \, ,
   \end{equation}
   yielding
   \begin{equation*}
   	\| D Q_{\xi,K,n_0} \|_{L^{1,\infty}}
   	\le 
   	C\frac{\lambda}{\mu r_\perp} \|  \frac{1}{K} \sum_{k=n_0+1}^{K+n_0} \Omega^k \|_{L^{1,\infty}}
   	\le 
   	C^2 \frac{ \lambda}{ \mu r_\perp K}
   	\, .\qedhere
   \end{equation*}
\end{proof}

\begin{remark}\label{rmk:improved Lorentz}
	It is not hard to prove the following extension of Lemma \ref{lemma: Lorentz}. For any $q\ge 1$ it holds
	\begin{equation}
		\| D W_{\xi,K,n_0} \|_{L^{1,q}} \le C\frac{ \lambda}{K^{1/2 - 1/q}} \left(\frac{ r_\parallel}{r_\perp} \right)^{1/2}\, ,
	\end{equation}
	\begin{equation}
		\| D Q_{\xi,K,n_0} \|_{L^{1,\infty}} \le C \frac{\lambda}{\mu r_\perp K^{1-1/q}} \, .
	\end{equation}
\end{remark}

\begin{lemma}\label{lemma: time correction}
	There exists a smooth $\lambda$-periodic function $A_{\xi, K, n_0} : \TT^2 \to \text{Sym}_2$ such that
	\begin{equation}\label{z100}
		\partial_t W_{\xi,K,n_0} = \div(A_{\xi,K,n_0}) \, ,
	\end{equation}
    \begin{equation}\label{z11}
    	\| A_{\xi,K,n_0} \|_{L^1} \le C(|\xi|,\Phi,\psi)  \mu K^{1/2} r_\perp^{3/2} r_{\parallel}^{-1/2} \, .
%    	
%    	 (r^{-1}_\parallel r_\perp) (r_\parallel r_\perp)^{1/2} \, .
    \end{equation}
\end{lemma}

\begin{proof}
	Setting 
	\begin{equation}
		A_{(\xi),k} 
		:=  -2^k \left(\frac{r_\perp}{r_\parallel}\right)\mu  \left(  \left(\frac{\xi}{|\xi|} \otimes  \frac{\xi^{\perp}}{|\xi|} + \frac{\xi^{\perp}}{|\xi|}\otimes \frac{\xi}{|\xi|} \right)    (\psi')_{(\xi)}^k (\Phi'')_{(\xi)}^k + \frac{r_\perp}{r_\parallel} \frac{\xi^\perp}{|\xi|} \otimes \frac{\xi^\perp}{|\xi|}  (\psi'')_{(\xi)}^k (\Phi')_{(\xi)}^k \right) \, ,
	\end{equation}
	\begin{equation}
		A_{(\xi),k}^c := 2^k	\left(\frac{r_\perp}{r_\parallel}\right)^2 \mu 
		\left(  \left(\frac{\xi}{|\xi|} \otimes  \frac{\xi^{\perp}}{|\xi|} + \frac{\xi^{\perp}}{|\xi|}\otimes \frac{\xi}{|\xi|}\right)    (\psi')_{(\xi)}^k (\Phi')_{(\xi)}^k - \frac{r_\perp}{r_\parallel}\frac{\xi^\perp}{|\xi|} \otimes \frac{\xi^\perp}{|\xi|}  (\psi'')_{(\xi)}^k (\Phi)_{(\xi)}^k  \right) \, ,
	\end{equation}
    it holds
    \begin{align*}
    	\partial_t W_{(\xi)}^k 
    	&= 2^{2k} \mu \lambda r_{\parallel}^{-1} \xi ( \psi')_{(\xi)}^k(x, t) \phi_{(\xi)}^k(x) 
    	\\& = 
    	-2^k\mu r_\parallel^{-1} r_\perp
    	\div\left(   \left(\frac{\xi}{|\xi|} \otimes  \frac{\xi^{\perp}}{|\xi|} + \frac{\xi^{\perp}}{|\xi|}\otimes \frac{\xi}{|\xi|}\right)    (\psi')_{(\xi)}^k (\phi'')_{(\xi)}^k + \frac{\xi^\perp}{|\xi|} \otimes \frac{\xi^\perp}{|\xi|}  (\psi'')_{(\xi)}^k (\phi')_{(\xi)}^k  \right)
    	\\&
    	=\div( A_{(\xi),k} ) \, ,
    \end{align*}  
and similarly    \begin{align*}
     	\partial_t (W_{(\xi)}^k)^c 
     	&=  \frac{r_\perp}{r_\parallel} 2^{2k} \mu \lambda r_{\parallel}^{-1} \xi ( \psi')_{(\xi)}^k(x, t) (\Phi'')_{(\xi)}^k(x) 
%     	\\& = 
%     	2^k\left(\frac{r_\perp}{r_\parallel}\right)^2 \mu \,
%     	\div\left(  \left(\frac{\xi}{|\xi|} \otimes  \frac{\xi^{\perp}}{|\xi|} + \frac{\xi^{\perp}}{|\xi|}\otimes \frac{\xi}{|\xi|}\right)    (\psi')_{(\xi)}^k (\Phi')_{(\xi)}^k - \frac{\xi^\perp}{|\xi|} \otimes \frac{\xi^\perp}{|\xi|}  (\psi'')_{(\xi)}^k (\Phi)_{(\xi)}^k  \right)
%     	\\&
     	=\div( A_{(\xi),k}^c ) \, .
     \end{align*}  
     Hence \eqref{z100} is satisfied. Defining
     \begin{equation}
     	A_{\xi,K,n_0} 
     	:= \frac{1}{K^{1/2}} \sum_{k=n_0+1}^{K + n_0} (A_{(\xi),k} + A_{(\xi),k}^c) \, 
     \end{equation}
and arguing as in Lemma \ref{lemma: first estimates}, we obtain that
      \begin{equation}
      	\| A_{(\xi),k} \|_{L^1}  + \| A_{(\xi),k}^c \|_{L^1}
      	\le C(|\xi|,\Phi,\psi) \mu K^{1/2} r_{\perp} r_\parallel^{-1} (r_\perp r_\parallel)^{1/2} \, ,
      \end{equation}
      which yields  \eqref{z11}.
\end{proof}

\subsection{Combinatorial lemma}\label{subsec:combinatorial}
The following proposition shows that, up to a suitable (time) translation of each element in the bundle, the building blocks associated to different directions can be taken disjoint.

\begin{propos}\label{prop: comb}
	Let $\xi_1=e_1$, $\xi_2 = e_2$, $\xi_3= e_1+e_2$ and $\xi_4= e_1-e_2$.
	Then for $n_0= 5K$ the functions in the family $\{ W_{(\xi_{i+1})}^k(x ,t+ i \mu^{-1} 2^{-5K})\}_{k=n_0,\ldots,n_0+K;\; i=0,1,2,3}$ have all supports mutually disjoint in space-time.
\end{propos}
\begin{proof}
	We apply Lemma \ref{lemma: comb} below to the families $\{ W_{(\xi_2)}^k(x ,t+ i \mu^{-1} 2^{-5K})\}_{k=n_0,...,n_0+K}$ and $\{ W_{(\xi_2)}^k(x,t+ j \mu^{-1} 2^{-5K})\}_{k=n_0,...,n_0+K}$; up to shifting the time axis, we can assume that $i=0$ and that $j\in \{1,2,3\}$ and conclude the proof.\end{proof}

\begin{lemma}\label{lemma: comb}
	Let $\xi_1, \xi_2 \in \{ e_1,e_2, e_1+e_2, e_1-e_2\} $ be two different vector fields. % e_2$, $\xi_3= ...$ and ...
	Let us consider two families $\{ W_{(\xi_1)}^k(x,t)\}_{k=n_0,\ldots,n_0+K}$ and $\{ W_{(\xi_2)}^k(x ,t+t_0)\}_{k=n_0,\ldots,n_0+K}$ for some $t_0 \in [\mu^{-1} 2^{-7K}, \mu^{-1} 2^{-7K+2}] $ and for $n_0=5K$.
	Then the supports of all these functions are disjoint in space-time, namely
	$$
	W_{(\xi_1)}^k(x,t) \otimes W_{(\xi_2)}^h(x,t+t_0)=0 \qquad \mbox{for all } k, h \in \{1,..., K\} .
	$$ 
\end{lemma}

\begin{proof}
	The family $\{ W_{(\xi_1)}^k(x,t)\}_{k=n_0,\ldots,n_0+K}$ is supported by \eqref{eqn:thin-tube} in space in a tube along $\xi_1$ of size $r_\parallel2^{-n_0}$ and similarly the family $\{ W_{(\xi_2)}^k(x,t+t_0)\}_{k=n_0,\ldots,n_0+K}$ is supported in the tube along  $\xi_2$ of size $r_\parallel2^{-n_0}$. Since these two thin tubes intersect only in a neighborhood of the origin, we deduce that the supports of $W_{(\xi_1)}^k(x,t)$ and $W_{(\xi_2)}^h(x,t)$, where $h, k \in \{n_0,\ldots, n_0+K\}$, can intersect for some time $t>0$ only if they both belong to $B_{R}(0)$, where $R:= r_\parallel 2^{-n_0 + 1}$.
	
	We claim the following: {\it suppose that for a certain $t>0$ and $k \in \{n_0,\ldots, n_0+K\}$ we have $\supp  W_{(\xi_1)}^k (\cdot, t) \cap B_{r_\parallel2^{-n_0+1}} \neq \emptyset$. Then $\supp  W_{(\xi_2)}^h(\cdot, t+t_0) \cap B_R = \emptyset$ for every $h \in \{n_0,\ldots, n_0+K\}$}. 
	
	The previous claim excludes the simultaneous presence at any $t>0$ of the support of $W_{(\xi_1)}^k (\cdot, t)$ and the support of $W_{(\xi_2)}^h(\cdot, t+t_0)$ in $B_R(0)$, thereby concluding the proof of the lemma.
	
	We now prove the claim.
	Let us fix a time $t$ such that  $\supp  W_{(\xi_1)}^k(\cdot,t) \cap B_{R} \neq \emptyset$. Since $\supp  W_{(\xi_1)}^k(\cdot,t) $ is moving at constant speed $\mu 2^k$ along the tube on the torus, there exists $\bar t$
	such that $|t -\bar t | \leq R \mu^{-1} 2^{-k}$ 
	and $\supp  W_{(\xi_1)}^k(\cdot,\bar t) %\cap B_{r_\parallel2^{-N_0+1}}
	= \supp  W_{(\xi_1)}^k(x,0)$. At time $\bar t$ we have information about the position of $\supp  W_{(\xi_2)}^h(\cdot,\bar t+t_0)$; more precisely,  we have that 
	\begin{equation}
		\label{eqn:union-ball}
		\supp  W_{(\xi_2)}^h(\cdot,\bar t+t_0) \subseteq \bigcup_{n\in \NN} \Big(\supp  W_{(\xi_2)}^h(\cdot,t_0) + n \frac{\xi_2}{2^K}\Big)
	\end{equation}
	because the ratio between the (constant) velocity of $\supp  W_{(\xi_1)}^k(\cdot,t)$ and the velocity of $\supp  W_{(\xi_2)}^k(\cdot,t)$ is of the form $2^j$ for some $j \in \{-K,..., K\}$.

	In the union in the right-hand side of \eqref{eqn:union-ball}, thanks to the upper bound on $t_0$, the choice $n=0$ identifies the ball of the (finite) union at minimal distance from the origin for every $k$. By the lower bound on $t_0$ and the fact that the minimal velocity is $\mu 2^{n_0}$, we get that this distance is greater than $2^{n_0-7K}$.
	At time $t$ the distance between $\supp  W_{(\xi_2)}^h(\cdot, t+t_0)$ and $B_R(0)$ is therefore bigger than
	$$
	2^{n_0-7K} - |t-\bar t| \mu 2^h -R \geq 2^{n_0-7K} -  R2^{h-k} -R \geq 2^{n_0-7K} -  R2^{K} -R\geq 2^{n_0-7K} -  2^{-n_0+K+1} = 2^{-2K}- 2^{-4K+1}>0.
	$$
	This concludes the proof of the claim.	
\end{proof}

\subsection{Proof of Proposition \ref{lemma: building blocks}}
Let $\{ W_{(\xi_{i+1})}^k(x ,t+ i \mu^{-1} 2^{-5K})\}_{k=n_0,\ldots,n_0+K;\; i=0,1,2,3}$ be as in Proposition \ref{prop: comb}. Since $\supp W_{\xi_{i+1}}^k = \supp (W_{\xi_{i+1}}^k)^c = \supp Q_{\xi_{i+1}}^k$, by translating in time $(W_{\xi_{i+1}}^k)^c$ and $Q_{\xi_{i+1}}^k$ with $t_{k,i} := i \mu^{-1} 2^{-5K}$  we deduce that $W_{i+1}^p := W^p_{\xi_{i+1},K,n_0}$, $W_{i+1}^c := W^c_{\xi_{i+1},K,n_0}$, $Q_{i+1}:= Q_{\xi_{i+1},K,n_0}$
and $A_{i+1}:=A_{\xi,K,n_0}$ satisfy (v) in Lemma \ref{lemma: building blocks}. We refer the reader to Lemma \ref{lemma: time correction} for the construction of $A_{\xi,K,n_0}$. Properties (i) and (ii) in Lemma \ref{lemma: building blocks} are now immediate from \eqref{eq: 2}, \eqref{eq: 3} and Lemma \ref{lemma: time correction}. We are left with the proof of (iii) and (iv) in Lemma \ref{lemma: building blocks}. To do so we have to choose appropriately the parameters $\lambda, \mu, K, r_\perp$ and $r_\parallel$. Let $\delta < 1/2$ to be chosen later in terms of $\eps>0$, we set
\begin{equation}
	\lambda = \left(\frac{r_\perp}{r_\parallel}\right)^{-1/2} \delta^4
	\quad
	K= \left(\frac{r_\perp}{r_\parallel}\right)^{-2}\delta^4
	\quad
	\mu = (r_\perp r_\parallel)^{-1/2}\delta^{-1} \, ,
\end{equation}
leaving $r_\perp \ll r_\parallel \ll 1$ free.
From Lemma \ref{lemma: Lorentz}, Lemma \ref{lemma: time correction}, \eqref{eq: W L2}, \eqref{eq: Q L2} and \eqref{eq: W Lp} we deduce

\begin{equation}\label{zz4}
	\| D(W_i^c + W_i^p) \|_{L^{1,\infty}} \le C\frac{ \lambda}{K^{1/2}} \left(\frac{ r_\parallel}{r_\perp} \right)^{1/2}
	=  C \delta^2
	\, ,
\end{equation}

\begin{equation}\label{zz2}
	\| D Q_i \|_{L^{1,\infty}} 
	\le C \frac{\lambda}{\mu K r_\perp} 
	= C \delta \frac{r_\perp}{r_\parallel} 
	\le C \delta \, ,
\end{equation}

 \begin{equation}\label{eq:temporalerror}
	\|A_i \|_{L^1}
	\le C \mu K^{1/2} (r^{-1}_\parallel r_\perp) (r_\parallel r_\perp)^{1/2}
	= C \delta \, ,
\end{equation}

\begin{equation}\label{zz3}
	\| Q_i \|_{L^2(\TT^2)} \le \frac{C}{\mu(r_\parallel r_\perp)^{1/2}}
	= C\delta \, ,
\end{equation}

\begin{equation}
	\| W_i^p \|_{L^2} + \frac{ r_\parallel}{r_\perp} \| W_i^c \|_{L^2} \le 1 \, .
\end{equation}
The conclusions (iii) and (iv) in Lemma \ref{lemma: building blocks} follow by choosing first $\delta$ small enough so that $C\delta \le \eps$, and after $r_\perp \ll r_\parallel\ll 1$ so that
$\frac{r_\perp}{r_\parallel} \le \eps$ and $\lambda = \delta^4 r_\parallel^{1/2} r_\perp^{-1/2} \ge \eps^{-1}$.

\section{Definition of the perturbations}\label{sec:def perturbation}

Let us begin by observing that there exist $\Gamma_i \in  C^\infty(\text{Sym}_2, \mathbb{R})$, $i=1,\ldots, 4$ such that
	\begin{equation}
		R = \sum_{i=1}^4 \Gamma_i(R)^2 e_i \otimes e_i\, ,
		\quad \text{for any $R\in \textit{Sym}_2$ such that $|R-I|<1/8$} \, ,
	\end{equation}
	where $e_1 :=(1,0)$, $e_2:=(0,1)$, $e_3:= (1/\sqrt{2}, 1/\sqrt{2})$ and $e_4:= (1/\sqrt{2},-1/\sqrt{2})$.

   We can define, for instance,
	\begin{equation}
		\Gamma_1(R)^2 := R_{1,1}-R_{1,2}-\frac{1}{2} \, ,
		\quad
		\Gamma_2(R)^2 := R_{2,2} - R_{1,2} - \frac{1}{2} \, ,
		\quad
		\Gamma_3(R)^2 := 2R_{1,2} + \frac{1}{2} \, ,
		\quad
		\Gamma_4(R)^2 := \frac{1}{2} \, .
	\end{equation}
	It is immediate to show the identity $R= \sum_{i=1}^4 \Gamma_i(R)^2 e_i \otimes e_i$. Moreover, using that $|R-I|<1/8$, we deduce
	\begin{equation}
		\Gamma_1(R)^2 
		= \frac{1}{2} + (R_{1,1}-1) - R_{1,2}
		\ge \frac{1}{2} - |R_{1,1}-1| - |R_{1,2}|
		\ge \frac{1}{4} \, ,
	\end{equation}
	\begin{equation}
		\Gamma_2(R)^2 
		= \frac{1}{2} + (R_{2,2}-1) - R_{1,2}
		\ge \frac{1}{2} - |R_{2,2}-1| - |R_{1,2}|
		\ge \frac{1}{4} \, ,
	\end{equation}
	\begin{equation}
		\Gamma_3(R)^2 \ge \frac{1}{2} - 2| R_{1,2}| \ge 1/4 \, ,
	\end{equation}
	which implies that $\Gamma_i$ are smooth functions.

  \bigskip

We define
\begin{equation}
	a_{i}(t,x) := (10\chi(t)(|R_0(t,x)| + \| R_0 \|_{L^1}))^{1/2}
	\, \Gamma_i \Big(I - \frac{10^{-1}}{|R_0(t,x)| + \| R_0 \|_{L^\infty(L^1)}} R(t,x) \Big) \, ,
\end{equation}
where $\chi\in C^\infty(\RR)$ satisfies $0\le \chi \le 1$, $\chi = 0$ on $[0,t_0/2]$, and $\chi = 1$ on $[t_0, \infty)$. Our choice leads to 
\begin{equation}
	\sum_{i=1}^4 a_i(t,x)^2 \frac{\xi_i}{|\xi_i|} \otimes \frac{\xi_i}{|\xi_i|}
	=
	-R_0(t,x) + \chi(t) 10(|R_0(t,x)| + \| R_0 \|_{L^\infty(L^1)}) I \, ,
\end{equation}
where $\xi_1=(1,0)$, $\xi_2=(0,2)$, $\xi_3=(1,1)$ and $\xi_4=(1,-1)$. The latter implies that
\begin{equation}\label{eq: quadratic identity}
	-\div(R_0) = \div\Big(   	\sum_{i=1}^4 a_i(x,t)^2 \frac{\xi_i}{|\xi_i|} \otimes \frac{\xi_i}{|\xi_i|}   \Big) + \nabla P \, ,
\end{equation}
for some pressure term $P$.

We observe that the coefficient $a_i$ is a ``slow function'', namely its derivatives are estimated only in terms of the smootness of $R_0$
\begin{equation}
	\|\partial_t^M \nabla^N a_i\|_{L^\infty}
	\le C(t_0, \| R_0 \|_{C^{N+M}}, N, M) \, ,
\end{equation}
\begin{equation}
	\| a_i \|_{L^\infty(L^2)} \le 5 \| R_0 \|_{L^\infty(L^1)}^{1/2} \, . 
\end{equation}

\medskip

For $\eps >0$ to be chosen later, we consider the functions $W_i^p$, $W_i^c$, $Q_i$, $A_i$ from Proposition~\ref{lemma: building blocks}.
We define the new velocity field as the sum of the previous one, a principal perturbation, a divergence corrector and a temporal corrector   
    \begin{equation}
    	u_{1} := u_0 + u_{1}^{(p)} + u_{1}^{(c)} + u_{1}^{(t)} \, ,
    \end{equation}
    where
    \begin{equation}
    	u_{1}^{(p)} = \sum_{i=1}^4a_i(W_{i} + W_{i}^c) \, ,
    	\quad
    	u_{1}^{(c)} = - \sum_{i=1}^4 \mathcal{R}(\nabla a_i \cdot (W_{i}^p + W_{i}^c)) \, ,
    	\quad
    		u_{1}^{(t)} = -\mathbb{P}(\sum_{i=1}^4 a_i^2 Q_{i}) \, .
    \end{equation}
We refer the reader to Remark \ref{rmk:scalar} for the definition of $\mathcal{R}$.

From now on, in order to simplify our notation, for any
function space $X$ and any map $f$ which depends on $t$ and $x$, we will write $\| f \|_X$ meaning $\| f \|_{L^\infty(X)}$.

\subsection{Estimate on $\| u_{1} - u_0 \|_{L^2}$ and on $\| u_{1} - u_0 \|_{L^1}$}\label{subs:u1-u0}
By the triangular inequality, 
$$
\| u_{1} - u_0 \|_{L^2} \leq \|u_{1}^{(p)}  \|_{L^2} + \|u_{1}^{(c)} \|_{L^2}  + \|u_{1}^{(t)} \|_{L^2}  
$$
and we estimate the right-hand side separately as
\begin{align}
\label{eqn:u1-u0L2_1}
\|u_{1}^{(p)}  \|_{L^2} 
& \leq \sum_{i=1}^4 \| a_i (W_{i}^p + W_{i}^c) \|_{L^2} 
\\& \leq \sum_{i=1}^4\left(    \| a_i \|_{L^2} \| W_{i}^p + W_{i}^c\|_{L^2} + C \frac{ \|a_i \|_{C_1} \| W_{i}^p + W_{i}^c\|_{L^2}}{\lambda^{1/2}}  \right) 
\\& \le \| R_0 \|_{L^1} + \eps^{1/2}C(t_0, \| R_0 \|_{C_1}) \, ,
\end{align}
where in the second line we used the improved Holder inequality \eqref{eqn:impr-holder} and (iii) in Proposition \ref{lemma: building blocks}.

From Remark \ref{rmk:scalar} we deduce
\begin{equation}
\label{eqn:u1-u0L2_2}
 \|u_{1}^{(c)} \|_{L^2}
 \leq 
 C\eps \sum_{i=1}^4 \|a_i \|_{C_2} \| W_{i}^p + W_{i}^c\|_{L^2}
 \le \eps C(t_0, \|R_0\|_{C^2})
  \, .
\end{equation}
Finally we employ (iv) in Proposition \ref{lemma: building blocks} to get

\begin{equation}
\label{eqn:u1-u0L2_3}
\|u_{1}^{(t)} \|_{L^2}
\leq
\sum_{i=1}^4 \|a_i \|_{L^\infty} \| Q_{i} \|_{L^2} 
\le \eps C(t_0, \| R_0\|_{L^\infty}) \, .
\end{equation}

Analogously
\begin{align}\label{eqn:u1-u0L1}
  \| u_{1} - u_0 \|_{L^1 }& \le
  \sum_{i=1}^4 \left( \| u_1^{(p)} \|_{L^1} +  \| u^{(c)}_1 \|_{L^1} + \| u^{(t)}_1 \|_{L^1}\right)
  \\& \le 
  C\sum_{i=1}^4   (\| a_i \|_{L^\infty} \| W_{i}^p + W_{i}^c\|_{L^1} +  \| u^{(c)}_1 \|_{L^2} + \| u^{(t)}_1 \|_{L^2} )
  \\& \le
   \eps C(t_0, \|R_0\|_{C^2}) \, .
\end{align}

\subsection{Estimate on $\| \curl(u_{1} - u_0 )\|_{L^{1,\infty}}$}
By triangular inequality,
\begin{align*}
	\| & \curl(u_{1} - u_0) \|_{L^{1,\infty}} \\& \le
	C\sum_{i=1}^4 \left( \| D(  a_i(W_{i}^p + W_{i}^c)) \|_{L^{1,\infty}} +  \| D\mathcal{R}(\nabla a_i \cdot (W_{i}^p + W_{i}^c) \|_{L^{1,\infty}} + \| \text{curl}\mathbb{P}(a_i Q_{i}) \|_{L^{1,\infty}}\right) \, ,
\end{align*}
we estimate the right-hand side separately as
\begin{equation}
	\| D(  a_i(W_{i}^p + W_{i}^c)) \|_{L^{1,\infty}}
	\le
    \| a_i \|_{C_1} \| (W_{i}^p + W_i^c) \|_{L^1} + \| a_i\|_{L^\infty} \| D (W_{i}^p + W_i^c) \|_{L^{1,\infty}}
    \le 
    \eps C(t_0, \| R_0 \|_{C^1}) \, ,
\end{equation}
\begin{equation*}
	\| \text{curl} \, \mathbb{P}(a_i Q_{i}) \|_{L^{1,\infty}}
	=
	\| \text{curl}(a_i Q_{i}) \|_{L^{1,\infty}}
	\le C \| a_i \|_{C_1} \| Q_{i} \|_{L^1} + \| a_i\|_{L^\infty} \| D Q_{i} \|_{L^{1,\infty}} 
	\le \eps C(t_0, \| R_0\|_{C^1})\, ,
\end{equation*}
where we employed (iv) in Proposition \ref{lemma: building blocks}.
Using that $D\mathcal{R}$ is a Calderon-Zygmund operator we deduce
\begin{equation}
	\| D\mathcal{R}(\nabla a_i \cdot (W_{i}^p + W_{i}^c)) \|_{L^{1,\infty}} 
	\le C \| \nabla a_i \cdot (W_{i}^p + W_{i}^c) \|_{L^1}
	\le  \eps C(t_0, \| R_0 \|_{C^1})\,  .
\end{equation}

\section{New error}\label{sec:new error}
We define $R_1$ in such a way that
\begin{equation}
\label{eqn:R1}
\partial_t u_1 + \div ( u_1 \otimes u_1) + \nabla p_1 = \div( R_1) \, ,
\end{equation}
which, by subtracting the equation for $u_0$, is equivalent to
\begin{equation}\label{z12}
	\div(R_1) = \div( u_0 \otimes (u_1-u_0) +  (u_1-u_0) \otimes u_0 + (u_1-u_0) \otimes (u_1-u_0) + R_0) + \partial_t (u_1 - u_0) + \nabla (p_1 - p_0) \, .
\end{equation}
We are going to define 
$$
R_1 := R_1^{(l)} + R_1^{(t)} + R_1^{(q)} \, ,
$$
where the various addends are defined in the following paragraphs, and show that
\begin{equation}\label{z13}
	\| R^{(l)}_1\|_{L^1} + \| R^{(t)} \|_{L^1} +  \| R^{(q)} \|_{L^1} \le \eps C(t_0, \| R_0\|_{C^3}) \, .
\end{equation}
The proof of Proposition \ref{prop: iterative} will follow by choosing $\eps$ small enough.

\subsection{Linear error} Let us set
\begin{equation}\label{eq:linear error}
	R^{(l)}_1 :=  u_0 \otimes (u_1-u_0) +  (u_1-u_0) \otimes u_0 \, ,
\end{equation}
thanks to \eqref{eqn:u1-u0L1} it holds
$$
\| R_1^{(l)}\|_{L^1}
\leq 2\|u_0\|_{L^\infty} \|u_1-u_0\|_{L^1} \leq \eps C(t_0, \| R_0 \|_{C^2}) \, .
$$

\subsection{Temporal error} 
Let us set
\begin{align*}
	R^{(t)}_1 & := \mathcal{R}( \partial_t a_i \cdot (W_{i}^p + W_{i}^c)) +  a_i A_i - \mathcal{R}( \nabla a_i \cdot A_{i} ) 
	\\& \quad + \mathcal{R}_0 \mathcal{R}( \partial_t(\nabla a_i) \cdot (W_{i}^p + W_{i}^c))
	+ \mathcal{R}_0 \mathcal{R}(\nabla a_i \cdot \div(A_i))
	-\mathcal{R}_0\mathbb{P}(\sum_{i=1}^4 \partial_t a_i^2 Q_{i}) \, .
\end{align*}
Using that
\begin{equation}
	\partial_t u_1^{(t)} 
	=
	-\mathbb{P}(\sum_{i=1}^4 \partial_t a_i^2 Q_{i}) 
	-
	\mathbb{P}(\sum_{i=1}^4 a_i^2 \div(W_i^p \otimes W^p_i))
	=
	-\mathbb{P}(\sum_{i=1}^4 \partial_t a_i^2 Q_{i}) 
	-
	\sum_{i=1}^4 a_i^2 \div( W_i^p \otimes W^p_i)
	- \nabla P \, ,
\end{equation}
for some pressure term $P$, it is immediate to verify the identity
\begin{equation}\label{eq: temporal}
	\partial_t (u_1 - u_0) = \div(R^{(t)}_1) - \sum_{i=1}^4 a_i^2 \div( W_i^p \otimes W^p_i)
	- \nabla P \, .
\end{equation}
Since $\mathcal R$ and $\mathcal{R}_0$ send $L^1$ to $L^1$ (cf. Lemma \ref{lemma23} and Remark \ref{rmk:scalar}), we have that
\begin{align*}
	\| \mathcal{R}( \partial_t a_i \cdot (W_{i}^p + W_{i}^c)) \|_{L^1} + \|  \mathcal{R}_0\mathcal{R}( \partial_t\nabla a_i \cdot (W_{i}^p + W_{i}^c))\|_{L^1} \leq 2\| a\|_{C^2} \|W_{i}^p + W_{i}^c \|_{L^1} \leq \eps C(t_0, \| R_0 \|_{C^2}) \, .
\end{align*}
\begin{equation}
	\| \mathcal{R}_0\mathbb{P}(\sum_{i=1}^4 \partial_t a_i^2 Q_{i}) \|_{L^1}
	\le\sum_{i=1}^4 \| \partial_t a_i^2 Q_{i} \|_{L^2}
	\le 
	\sum_{i=1}^4 \| \partial_t a_i^2 \|_{L^\infty}\| Q_{i} \|_{L^2}
	\le \eps C(t_0, \| R_0 \|_{C^1}) \, .
\end{equation}
From (iv) in Proposition \ref{lemma: building blocks} we get
\begin{equation}
	\|  a_i A_i\|_{L^1} + \|\mathcal{R}( \nabla a_i \cdot A_{i} ) \|_{L^1}
	\le 2\| a_i \|_{C^1} \| A_i \|_{L^1}
	\le \eps C(t_0, \| R_0\|_{C^1}) \, .
\end{equation}
By employing \eqref{eqn:R-di-div} we bound
\begin{equation}
	\| \mathcal{R}_0\mathcal{R}(\nabla a_i \cdot \div(A_i)) \|_{L^1}
	\le C\| a_i \|_{C^3} \| A_i \|_{L^1}
	\le \eps C(t_0, \|R_0\|_{C^3})  \, .
\end{equation}

\subsection{Quadratic error terms}
Let us set
\begin{equation}
	R_1^{(q)} 
	= (u_1 - u_0)\otimes (u_1 - u_0) - \sum_{i=1}^4
	a_i^2 W_{i}^p\otimes W_{i}^p 
	+
	\sum_{i=1}^4\mathcal{R}\Big( \nabla a_i^2 \cdot \Big(W_i^p \otimes W^p_i  - \int_{\TT^2} W^p_i \otimes W_i^p\Big) \Big) \, ,
\end{equation}
and show that \eqref{z12} holds. In view of \eqref{eq:linear error}, \eqref{eq: quadratic identity} and \eqref{eq: temporal} it amounts to check that
\begin{equation}
	\div(R_1^{(q)}) =  \div \Big( (u_1-u_0) \otimes (u_1-u_0) - \sum_{i=1}^4 a_i^2 \Big(\frac{\xi_i}{|\xi_i|} \otimes \frac{\xi_i}{|\xi_i|} \Big)\Big) - \sum_{i=1}^4 a_i^2 \div( W_i^p \otimes W^p_i)  + \nabla (p_1 - p_2) \, .
\end{equation}
The latter easily follows by noticing that, as a consequence of (ii) in Proposition \ref{lemma: building blocks}, one has
\begin{align*}
	\sum_{i=1}^4 \nabla a_i^2 \cdot \Big(W_i^p \otimes W^p_i  - \int_{\TT^2} W^p_i \otimes W_i^p \Big)
	& =
		\sum_{i=1}^4 \nabla a_i^2 \cdot \Big(W_i^p \otimes W^p_i  - \int_{\TT^2} W^p_i \otimes W_i^p \Big) 
	\\& =
	\sum_{i=1}^4   \div\Big( a_i^2 \Big(W_i^p \otimes W^p_i  - \frac{\xi_i}{|\xi_i|} \otimes \frac{\xi_i}{|\xi_i|} \Big)\Big) - \sum_{i=1}^4 \div(a_i^2 W_i^p \otimes W^p_i) \, .
\end{align*}

\bigskip

Let us finally prove that $\| R_1^{(q)}\|_{L^1} \le \eps C(t_0, \| R_0\|_{C^2})$. We begin by observing that
\begin{align*}
	& (u_1 - u_0) \otimes  (u_1-u_0) - \sum_{i=1}^4
	a_i^2 W_{i}^p\otimes W_{i}^p 
	\\& =
	\sum_{i=1}^4 (
	a_i^2 W_{i}^p\otimes W_{i}^c + a_i^2W_{i}^c\otimes W_{i}^p +
	a_i^2 W_{i}^c\otimes W_{i}^c) 
	 + ( u_{1}^{(c)} + u_{1}^{(t)}) \otimes ( u_1-u_0)+  ( u_1-u_0) \otimes ( u_{1}^{(c)} + u_{1}^{(t)}) \, ,
\end{align*}

From (iv) in Proposition \ref{lemma: building blocks}, the estimates in Section~\ref{subs:u1-u0} on $ \|  u_{1}^{(c)}\|_{L^2},\| u_{1}^{(t)} \|_{L^2} ,
\| u_1-u_0\|_{L^2}$  and Lemma~\ref{lemma23} we deduce
\begin{equation}
	\| 
	a_i^2 W_{i}^p\otimes W_{i}^c + a_i^2W_{i}^c\otimes W_{i}^p +
	a_i^2 W_{i}^c\otimes W_{i}^c \|_{L^1}
	\le \| a_i \|_{L^\infty}( 2 \|W_i^p\|_{L^2}\| W_i^c\|_{L^2} + \| W_i^c\|_{L^2}^2)
	\le \eps C(t_0, \| R_0\|_{L^\infty}) \, ,
\end{equation}
\begin{equation}
\begin{split}
\| ( u_{1}^{(c)} + u_{1}^{(t)}) \otimes ( u_1-u_0)+  ( u_1-u_0) \otimes ( u_{1}^{(c)} + u_{1}^{(t)})\|_{L^1} &\leq 2 \|  u_{1}^{(c)} + u_{1}^{(t)} \|_{L^2} 
\| u_1-u_0\|_{L^2}
\leq %(C\| R_0 \|_{L^1} + \eps^{1/2}C(t_0, \| R_0 \|_{C_1})  ) 
\eps C(t_0, \| R_0 \|_{C_2}),
\end{split}
\end{equation}
\begin{equation}
	\| \mathcal{R}( \nabla a_i^2 \cdot (W_i^p \otimes W^p_i  - \int_{\TT^2} W^p_i \otimes W_i^p) ) \|_{L^1}
	\le
	 C \eps \| \nabla a_1\|_{C^1} \|W_i^p \otimes W^p_i\|_{L^1}
	 \le \eps C(t_0, \| R_0 \|_{C^2}) \, .
\end{equation}

\bibliographystyle{plain}
\bibliography{NSE}

\begin{thebibliography}{10}

\bibitem{MR438106}
Angelo Alvino.
\newblock Sulla diseguaglianza di {S}obolev in spazi di {L}orentz.
\newblock {\em Boll. Un. Mat. Ital. A (5)}, 14(1):148--156, 1977.

\bibitem{BrSh21}
Alberto Bressan and Wen Shen.
\newblock A posteriori error estimates for self-similar solutions to the
  {E}uler equations.
\newblock {\em Discrete Contin. Dyn. Syst.}, 41(1):113--130, 2021.

\bibitem{BV}
T.~Buckmaster and V.~Vicol.
\newblock Nonuniqueness of weak solutions to the {N}avier-{S}tokes equation.
\newblock {\em arXiv preprint arXiv:1709.10033}, 2017.

\bibitem{BCV18}
Tristan Buckmaster, Maria Colombo, and Vlad Vicol.
\newblock Wild solutions of the navier-stokes equations whose singular sets in
  time have hausdorff dimension strictly less than 1.
\newblock {\em Journal of the European Mathematical Society}, to appear,
  arXiv:1809.00600.

\bibitem{Buckmaster:2021um}
Tristan Buckmaster, Nader Masmoudi, Matthew Novack, and Vlad Vicol.
\newblock Non-conservative $h^{\frac 12-}$ weak solutions of the incompressible
  3d euler equations.
\newblock 01 2021.

\bibitem{BSV16}
Tristan Buckmaster, Steve Shkoller, and Vlad Vicol.
\newblock Nonuniqueness of weak solutions to the {SQG} equation.
\newblock {\em Comm. Pure Appl. Math.}, 72(9):1809--1874, 2019.

\bibitem{MR3352460}
Elisabetta Chiodaroli, Camillo De~Lellis, and Ond\v{r}ej Kreml.
\newblock Global ill-posedness of the isentropic system of gas dynamics.
\newblock {\em Comm. Pure Appl. Math.}, 68(7):1157--1190, 2015.

\bibitem{CDRS21typicality}
Maria Colombo, Luigi~De Rosa, and Massimo Sorella.
\newblock Typicality results for weak solutions of the incompressible
  navier--stokes equations.
\newblock 02 2021.

\bibitem{DeLellisSzekelyhidi09}
C.~De~Lellis and L.~Sz{\'e}kelyhidi, Jr.
\newblock The {E}uler equations as a differential inclusion.
\newblock {\em Ann. of Math. (2)}, 170(3):1417--1436, 2009.

\bibitem{Delort91}
Jean-Marc Delort.
\newblock Existence de nappes de tourbillon en dimension deux.
\newblock {\em J. Amer. Math. Soc.}, 4(3):553--586, 1991.

\bibitem{DPM87}
Ronald~J. DiPerna and Andrew~J. Majda.
\newblock Concentrations in regularizations for {$2$}-{D} incompressible flow.
\newblock {\em Comm. Pure Appl. Math.}, 40(3):301--345, 1987.

\bibitem{EvansMuller94}
L.~C. Evans and S.~M\"{u}ller.
\newblock Hardy spaces and the two-dimensional {E}uler equations with
  nonnegative vorticity.
\newblock {\em J. Amer. Math. Soc.}, 7(1):199--219, 1994.

\bibitem{Grafakos}
L.~Grafakos.
\newblock {\em Classical and modern {F}ourier analysis}.
\newblock Pearson Education, Inc., Upper Saddle River, NJ, 2004.

\bibitem{Isett2018Annals}
P.~Isett.
\newblock A proof of {O}nsager's conjecture.
\newblock {\em Ann. of Math. (2)}, 188(3):871--963, 2018.

\bibitem{Yud62}
V.~I. Judovi\v{c}.
\newblock Some bounds for solutions of elliptic equations.
\newblock {\em Mat. Sb. (N.S.)}, 59 (101)(suppl.):229--244, 1962.

\bibitem{Yud63}
V.~I. Judovi\v{c}.
\newblock Non-stationary flows of an ideal incompressible fluid.
\newblock {\em \v{Z}. Vy\v{c}isl. Mat i Mat. Fiz.}, 3:1032--1066, 1963.

\bibitem{MR2246357}
Gr\'{e}goire Loeper.
\newblock Uniqueness of the solution to the {V}lasov-{P}oisson system with
  bounded density.
\newblock {\em J. Math. Pures Appl. (9)}, 86(1):68--79, 2006.

\bibitem{MoSz2019AnnPDE}
S.~Modena and L.~Sz\'{e}kelyhidi, Jr.
\newblock Non-uniqueness for the transport equation with {S}obolev vector
  fields.
\newblock {\em Ann. PDE}, 4(2):Art. 18, 38, 2018.

\bibitem{MR4138227}
Stefano Modena and Gabriel Sattig.
\newblock Convex integration solutions to the transport equation with full
  dimensional concentration.
\newblock {\em Ann. Inst. H. Poincar\'{e} Anal. Non Lin\'{e}aire},
  37(5):1075--1108, 2020.

\bibitem{Nash}
J.~Nash.
\newblock {$C^1$} isometric imbeddings.
\newblock {\em Ann. of Math. (2)}, 60:383--396, 1954.

\bibitem{DRT19typicality}
Luigi~De Rosa and Riccardo Tione.
\newblock Sharp energy regularity and typicality results for h{\"o}lder
  solutions of incompressible euler equations.
\newblock 08 2019.

\bibitem{Vis18a}
Misha Vishik.
\newblock Instability and non-uniqueness in the cauchy problem for the euler
  equations of an ideal incompressible fluid. part i.
\newblock {\em arxiv:1805.09426}.

\bibitem{Vis18}
Misha Vishik.
\newblock Instability and non-uniqueness in the cauchy problem for the euler
  equations of an ideal incompressible fluid. part ii.
\newblock {\em arxiv:1805.09440}.

\end{thebibliography}
\end{document}